\documentclass[12pt]{article} 

\usepackage{amsmath, amsfonts, amssymb}
\usepackage{mathrsfs}
\usepackage{theorem}
\usepackage{bm}
\usepackage[dvipdf]{graphicx}

\RequirePackage[colorlinks,citecolor=blue,urlcolor=blue,bookmarksopen]{hyperref}

\newtheorem{mythm}{Theorem}[section]

\newtheorem{mycor}[mythm]{Corollary}

{\newtheorem{myrem}[mythm]{Remark}}
{\newtheorem{myexam}{Example}[section]}
\newcommand{\dis}{\displaystyle}

\def\R{\mathbb R}

\def\N{\mathbb N}

\def\d{\mathrm{d}}
\def\E{\mathbb E}
\def\p{\mathbb P}
\def\e{\mathrm{e}}

\def\La{\Lambda}

\def\S{\mathcal S}

\def\wt{\widetilde}

\def\m{\mathbf{m}}
\def\var{\mathrm{var}}

\newcommand{\fin}{\hfill $\square$\par}
\newenvironment{proof}{{\noindent\it Proof.}\ }{\hfill $\square$\par}

\numberwithin{equation}{section}

\allowdisplaybreaks

\begin{document}

 \title{Comparison theorem and stability under perturbation of transition rate matrices for regime-switching processes \footnote{Supported in
			part by NNSFs of China (No. 12271397,   11831014) and National Key R\&D Program of China (No. 2022YFA1000033)}}

\author{Jinghai Shao\thanks{Center for Applied Mathematics, Tianjin University, Tianjin 300072, China. Email: shaojh@tju.edu.cn.}
}

\maketitle

\begin{abstract}
A comparison theorem for state-dependent regime-switching diffusion processes is established, which enables us to control pathwisely the evolution of the state-dependent switching component simply by Markov chains. Moreover, a sharp estimate on the  stability of Markovian regime-switching processes under the perturbation of transition rate matrices is provided. Our approach is based on the elaborate constructions of switching processes in the spirit of Skorokhod's representation theorem varying according to the problem being dealt with. In particular, this method can cope with the switching processes in an infinite state space and not necessarily being of birth-death type.  As an application, some known results on ergodicity and stability of state-dependent regime-switching processes can be improved.
\end{abstract}

\textbf{keywords} Comparison theorem, Regime-switching diffusions, Ergodicity, Perturbation theory


\section{Introduction}

Stochastic processes with regime-switching have been extensively studied in many research fields due to their characterization of random changes of the environment between different regimes; see, for instance, \cite{CH,Gho,MY,NYZ,Sh15,Sh15a,SX19,YZ,Zh19} and references therein. Especially, when the switching of the environment depends on the state of the studied dynamic system, called usually state-dependent regime-switching process(for short, \textbf{RSP}), the properties of such system become much more complicated due to their intensive interaction.
The monograph \cite{YZ} has introduced various properties of the state-dependent \textbf{RSP}, which tells us that it is a very challenging task to provide easily verifiable conditions to justify the ergodicity and stability of the state-dependent \textbf{RSP}s.

In view of the relatively abundant results on state-independent \textbf{RSP}s, it is natural to simplify a state-dependent \textbf{RSP} into a state-independent one. So, one needs to establish some kind of comparison theorem for state-dependent \textbf{RSP}s. Such idea has been used in many works. For instance, Cloez and Hairer \cite{CH} used this idea to study the exponential ergodicity in the Wasserstein distance for birth-death type state-dependent \textbf{RSP}s based on an application of weak Harris' theorem. The state-dependent switching process and the constructed state-independent process constitute a coupling process. However, the constructed coupling process in \cite{CH} is not a Markovian process any longer and needs modification in applications (cf. \cite[Remark 3.10]{CH}).  Majda and Tong used the same method to study the exponential ergodicity in the setting of piecewise deterministic processes with regime-switching and applied their results to the tropical stochastic lattice models. In \cite{Sh18}, the author constructed a Markovian coupling for the birth-death type state-dependent \textbf{RSP}s, then, in \cite{SX19}, together with Xi, extended it to a general case under the condition of the existence of an order-preserving coupling for Markov chains(cf. \cite[Lemma 2.7]{SX19}). This result is not constructive, and the verification of Assumption 2.6 in \cite{SX19} is not easy in application.

Accordingly, our first purpose is to establish a comparison theorem for more general state-dependent \textbf{RSP}s, especially to get rid of the restriction of  birth-death type switching and to be applicable to switching processes in an infinitely countable state space. Our comparison theorem is in the pathwise sense. As a consequence, the corresponding results in \cite{CH,MT,Sh18,SX19} can be generalized after certain necessary modification.

For state-dependent \textbf{RSP}s, the study of Feller property and smooth dependence of initial values is of great interest. It is quite different to that of Markovian \textbf{RSP}s and diffusion processes, which has been noted in the works \cite{NYZ}, \cite{Xi08}, \cite[Chapter 2]{YZ} and references therein. For instance, in \cite[Theorem 3.1]{NYZ}, the continuous  differentiability of continuous component of \textbf{RSP}s w.r.t. the initial value in $L^p$ with $p\!\in \!(0,1)$ was proved under suitable regular conditions. Moreover, restricted to a bounded set, certain gradient estimate associated with the continuous component of \textbf{RSP}s  can be derived from \cite[Theorem 4.1]{NYZ}. Recently, together with K. Zhao, we studied in \cite{SZ21} the continuous dependence of initial value for state-dependent \textbf{RSP}s, which are solutions to certain stochastic functional differential equations with regime-switching.  As shown in \cite{SZ21}, the key point is the estimate for the following quantity
\begin{equation}\label{quan}
\Theta(t, \La,\tilde \La):=\frac 1t\int_0^t \p(\La_s\neq \tilde \La_s)\d s,\qquad \quad t>0.
\end{equation}
Moreover, it was also shown in \cite{SY} that the quantity $\Theta(t,\La,\tilde \La)$ plays a crucial role in the study of  stability of \textbf{RSP}s under the perturbation of $Q$-matrix. However, the estimate in \cite{SY,SZ21} does not work for switching processes in an infinite state space, which needs to be generalized from the viewpoint of application. Besides, the estimation of $\Theta(t,\La,\tilde\La)$ develops the classical perturbation theory for Markov chains(cf. \cite{Mi03,Mi04,Mi05,ZI94}). See more details in Section 3 below.


Our improvements in the previous two concerned problems are all based on a new observation on Skorokhod's representation theorem for the jumping processes.
The approach  of using  Skorokhod's representation theorem to study \textbf{RSP}s has been widely used in the literature; see, e.g. \cite{Gho}, \cite{NYZ}, \cite{Sh15}, \cite{SY}, \cite{YZ} etc.  The basic idea of Skorokhod's representation theorem is to represent the switching process in terms of an integral with respect to a Poisson random measure based on a sequence of constructed intervals on the real line. The length of each interval  is determined by $(q_{ij}(x))$. However, the impact of the construction and sort order of the intervals on the obtained jumping processes has been neglected in all the previous works. In this work, we shall show  that it is necessary to carry out different construction of the intervals  to solve different problems. This will be illustrated through establishing the comparison theorem and studying the stability problem under the perturbation of $Q$-matrix.

This work is organized as follows. In Section 2, we first provide the construction of the coupling process, then establish the comparison theorem for state-dependent \textbf{RSP}s. As an illustrative example, we use this comparison theorem to study the stability of state-dependent \textbf{RSP}s.
In Section 3, we also first provide the construction of the coupling process $(\La_t,\tilde \La_t)$ in \eqref{quan}, the key point of which is the construction of intervals used in Skorokhod's representation theorem. Then we provide an estimate of \eqref{quan} which improves the one given in \cite{SY}.

\section{Comparison theorem for state-dependent \textbf{RSP}s}

Let $\S=\{1,2,\ldots,N\}$ with $2\leq N\leq \infty$. So, $\S$ is allowed to be an infinitely countable state space by taking $N=\infty$. Consider
\begin{equation}\label{a-1}
\d X_t=b(X_t,\La_t)\d t+\sigma(X_t,\La_t)\d B_t,
\end{equation} where $b:\R^d\times\S\to \R^d$, $\sigma:\R^d\times\S\to \R^{d\times d}$ satisfying suitable conditions and $(B_t)$ is a $d$-dimensional Brownian motion. Here $(\La_t)$ is a jumping process on $\S$ satisfying
\begin{equation}\label{a-2}
\p\big(\La_{t+\delta}=j|\La_t=i, X_t=x\big)=\begin{cases}
  q_{ij}(x)\delta +o(\delta), &\text{if $i\neq j$},\\
  1+q_{ii}(x)\delta +o(\delta), &\text{if $i=j$,}
\end{cases}
\end{equation} provided $\delta>0$. When $(q_{ij}(x))$ does not depend on $x$, $(X_t,\La_t)$ is called a state-independent \textbf{RSP} or a Markovian \textbf{RSP}. Meanwhile, when $(q_{ij}(x))$ does depend on $x$, $(X_t,\La_t)$ is called a state-dependent \textbf{RSP}, which is used to model the phenomenon that the dynamic system $(X_t)$ can impact the change rate of the random environment in applications. There is an intensive interaction between $(X_t)$ and $(\La_t)$ for the state-dependent \textbf{RSP}s, and hence it is quite desirable to develop a stochastic comparison theorem to simplify such system. Since $(X_t)$ can be simplified, if necessary, by using classical comparison theorem for diffusion processes(cf. \cite{IW77,IW}) or for L\'evy process (cf. \cite{Wa}) performed separatively in each fixed regime $i$, hence the key  point is to control the jumping component $(\La_t)$ whose transition rates vary with $(X_t)$.

\begin{mythm}\label{main-1}
Assume that SDEs \eqref{a-1} and \eqref{a-2} admit a solution $(X_t,\La_t)$ for any initial value $(X_0,\La_0)=(x,i)\in\R^d\times \S$. Suppose that $(q_{ij}(x))$ is conservative for every $x\in \R^d$ and satisfies
\begin{itemize}
  \item[$\mathrm{(Q1)}$] $K_0:=\sup_{x\in\R^d}\sup_{i\in\S} |q_{ii}(x)|<\infty$.
  \item[$\mathrm{(Q2)}$]   $\forall\, i\!\in\! \S$, there is a   $c_i \!\in \!\N$ such that $q_{ij}(x)=0$,    $\forall\,j\in \S$ with $|j\!-\! i|\!>c_i $, $\forall\, x\in \R^d$.
\end{itemize}
Define
\begin{equation*}
  q_{ij}^\ast=\begin{cases}
    \sup\limits_{x\in\R^d}\max\limits_{j<\ell\leq i}q_{\ell j}(x),\  &j<i\\
    \inf\limits_{x\in\R^d}\min\limits_{\ell\leq i} q_{\ell j}(x),\  &j>i\\
    -\sum_{i\neq j}q_{ij}^\ast, &j=i
  \end{cases}\ \  \text{and}\quad
  \bar q_{ij}=\begin{cases}
    \inf\limits_{x\in\R^d}\min\limits_{j<\ell\leq i} q_{\ell j}(x),\  &j<i\\
    \sup\limits_{x\in\R^d} \max\limits_{\ell\leq i} q_{\ell j}(x),\  &j>i\\
    -\sum_{j\neq i}\bar q_{ij},  &j=i
  \end{cases}.
\end{equation*}
Then there exist two continuous-time Markov chains $(\La^\ast_t)$ and $(\bar \La_t)$ on $\S$ with the transition rate matrix $(q_{ij}^\ast)$ and $(\bar q_{ij})$ respectively such that
\begin{equation}\label{a-3}
\p(\La_t^\ast\leq \La_t\leq \bar \La_t,\ \forall\,t\geq 0)=1.
\end{equation}
\end{mythm}

\begin{myrem}
  \begin{itemize}
  \item[$\mathrm{(i)}$] For birth-death type $(q_{ij}(x))$, i.e. $q_{ij}=0$ for $|i-j|\geq 2$, we have
  \begin{equation*}
  \begin{split}
    q_{i(i-1)}^\ast&=\sup_{x\in\R^d}q_{i(i-1)}(x),\
    q_{i(i+1)}^\ast =\inf_{x\in\R^d} q_{i(i+1)}(x),\\
    \bar q_{i(i-1)}&=    \inf_{x\in\R^d} q_{i(i-1)}(x),\
    \bar q_{i(i+1)}=\sup_{x\in\R^d} q_{i(i+1)}(x).
  \end{split}
  \end{equation*}
  This coincides with the Markov chain constructed in \cite{Sh18}.
  \item[$\mathrm{(ii)}$] There are many works to investigate the existence of solution to \eqref{a-1} and \eqref{a-2}; see, \cite{YZ} under Lipschitzian condition, \cite{Sh15} under non-Lipschitzian condition, \cite{Zh19} under integrable condition.
  \item[$\mathrm{(iii)}$]
      Assumption $\mathrm{(Q1)}$ ensures that there exists a unique Markov chain $(\La_t^\ast)_{t\geq 0}$ $((\bar \La_t)_{t\geq 0} )$ associated with $(q_{ij}^\ast)$ $($$(\bar q_{ij})$ respectively$)$; see, e.g.
      \cite[Corollary 2.24]{Chen}.
  \end{itemize}
\end{myrem}

As mentioned in the introduction, this comparison theorem can help us to generalize the corresponding results in \cite{CH}, \cite{MT}, \cite{Sh18}, \cite{SX19}. Precisely, one can remove Assumption 3.1 of birth-death type restriction in \cite{CH} and generalizes \cite[Theorems 3.3, 3.4]{CH} there. Also,  the exponential convergence results, \cite[Theorems 2.1, 2.3]{MT}, for two stochastic lattice models for moist tropical convection in climate science studied can be generalized by considering more general jumping processes besides birth-death processes used in \cite{MT}.

As an illustrative example, we
apply Theorem \ref{main-1} to investigate the stability of  state-dependent \textbf{RSP}s.
The stability of stochastic processes with regime-switching has been studied in many works. We refer the readers to the monographs \cite{MY,YZ} for surveys on this topic, and also to \cite{SX} for some recent results on the stability of state-dependent \textbf{RSP}s based on $M$-matrix theory, Perron-Frobenius theorem and the Friedholm alternative.

\begin{mythm}\label{app-1}
Let $(X_t^{x,i}, \La_t^{x,i})$ be the solution to \eqref{a-1} and \eqref{a-2} with initial value $(x,i)$. Assume that the conditions of Theorem \ref{main-1} hold and $\S$ is a finite state space. Suppose that there exist a function $\rho\in C^2(\R^d)$, constants $\beta_i\in\R$ for $i\in\S$, constants $p,\,\tilde c>0$ such that
\begin{equation}\label{g-9}
\begin{split}
\mathscr{L}^{(i)}\rho(x)
 \leq \beta_i \rho (x),\qquad \rho(x)\geq \tilde c|x|^p,\quad x\in\R^d,\ i\in \S,
\end{split}
\end{equation} where $\mathscr{L}^{(i)}\rho(x)\!=\! \sum\limits_{k=1}^d\!b_k(x,i)\partial_k\rho(x)\! +\frac 12\sum\limits_{k,l=1}^d\!\! a_{kl}(x,i)\partial_k\partial_l\rho(x)$, $a_{kl}(x,i)\!=\!\sum\limits_{j=1}^d\!  \sigma_{kj}(x,i)\sigma_{lj}(x,i)$. Through reordering $\S$, without loss of the generality we may assume that $(\beta_i)_{i\in\S}$ is  nondecreasing. Let $(\bar q_{ij})$ be defined as in Theorem \ref{main-1}. Assume  that $(\bar q_{ij})$ is irreducible, and
\begin{equation}\label{g-10}
\sum_{i\in\S}\bar\mu_i \beta_i<0
\end{equation}
where $(\bar\mu_i)_{i\in\S}$ denotes the unique invariant measure of $(\bar q_{ij})$.
Then there exists $p'\in (0,p]$ such that
\begin{equation}\label{g-11}
\lim_{t\to \infty} \E|X_t^{x,i}|^{p'}=0,\quad \quad x\in\R^d,\,i\in\S.
\end{equation}
\end{mythm}

\begin{myrem}
In Theorem \ref{app-1}, via condition \eqref{g-9}, we characterize the stability property of the process $(X_t)$ at each fixed state $i\in\S$ by a constant $\beta_i\in\R$, which is measured under a common Lyapunov  type function $\rho$. Then, under the help of the auxiliary Markov chain $(\bar q_{ij})$, the average condition \eqref{g-10} yields the stability of $(X_t)$.
\end{myrem}

Next, we provide an example to illustrate the application of Theorem \ref{app-1}.
\begin{myexam}\label{exam-1}
Consider
\[\d X_t=b_{\La_t}X_t\d t+\sigma_{\La_t} X_t \d B_t,\]
where $(\La_t)$ is a jumping process on $\S=\{1,2,3\}$ with the state-dependent transition rate matrix $(q_{ij}(x))_{i,j\in\S}$ given by
\[q_{ij}(x)=1+|i-j|(x^2\wedge 1),\quad i,j\in\S, \ x\in\R.\]

Take $\rho(x)=x^2$ in \eqref{g-9} to yield  $\beta_i=2b_i+\sigma_i^2$, $i\in\S$. By virtue of  Theorem \ref{main-1}, direct calculation yields that
\[(\bar q_{ij})_{i,j\in\S}=\begin{pmatrix}
  -5 &2 &3\\ 1&-4 & 3\\ 1&1&-2
\end{pmatrix},\]
and its unique invariant probability measure is  $(\bar\mu_1,\bar\mu_2,\bar\mu_3)=(\frac 16, \frac 7{30}, \frac{3}{5})$. Then, by Theorem \ref{app-1}, if
\[\sum_{i=1}^3\bar\mu_i\beta_i<0,\]
the process $(X_t)$ is stable in the sense that
\[\lim_{t\to \infty} \E\big[ |X_t^{x,i}|^p\big]=0,\quad x\in \R,\ i\in\S,\ \text{ for some $p\in (0,2]$}.\]
\end{myexam}

\subsection{Construction of the coupling process}

In the spirit of Skorokhod \cite{Sk89} in the study of processes with rapid switching, Ghosh et al. \cite{Gho} presented a representation of Markov chain in terms of a Poisson random measure and studied the stability of \textbf{RSP}s. This kind of representation theorem was widely applied in the study of various properties of \textbf{RSP}s.

Before giving out our precise representation theorem for establishing the comparison theorem, let us recall the construction in \cite{Gho} for comparison.  When $\S=\{1,2,\ldots,N\}$ is a finite state space, let $\Delta_{ij}(x)$ be a consecutive, left-closed, right-open intervals on $[0,\infty)$, each having length $q_{ij}(x)$.
More precisely,
\begin{gather*}\Delta_{12}(x)=[0, q_{12}(x)), \ \ \ \ \ \ \ \ \quad \Delta_{13}(x)=[q_{12}(x), q_{12}(x)+q_{13}(x)),\ldots, \\ \Delta_{1N}(x)=\big[\sum_{j\neq 1} q_{1j}(x), q_1(x)\big),\ \
\Delta_{21}(x)=[q_1(x), q_1(x)+q_{21}(x)),\ldots,
\end{gather*} and so on.
Define a function
$h:\R^d\times \S\times \R\to \R$ by
\begin{equation}\label{a-5} h(x,i,z)=\sum_{j\in\S,j\neq i}(j-i) \mathbf{1}_{\Delta_{ij}(x)}(z).
\end{equation}
Then, $(\La_t)$ is a jumping process satisfying \eqref{a-2} as a solution to the following SDE:
\begin{equation}\label{a-6}
\d \La_t=\int_{[0,\infty)} h(X_t,\La_{t-}, z)\mathcal{N}(\d t,\d z),
\end{equation}
where $(X_t)$ satisfies \eqref{a-1}, $\mathcal{N}(\d t,\d z)$ is a Poisson random measure with intensity $\d t\!\times\!\d z$.

Now we shall introduce our construction of the coupling process in three steps. We assume the conditions (Q1) (Q2) hold in this section.

\noindent\textbf{Step 1}. \emph{Construction of intervals} For every fixed $x\in\R^d$ and $i,j\in \S$, we define the intervals $\Gamma_{ij}(x)$, $\Gamma_{ij}^\ast$ and $\bar\Gamma_{ij}$ in the following way. Starting from $0$, the intervals $\Gamma_{ij}(x)$ for $j<i$ are defined on the positive half line, while $\Gamma_{ij}(x)$ for $j>i$ are defined on the negative half line. Precisely,
\begin{equation}\label{a-7}
\begin{array}{ll}
&\Gamma_{i1}(x)=[0, q_{i1}(x)),  \quad \Gamma_{i2}(x)=[q_{i1}(x), q_{i1}(x)\!+\!q_{i2}(x)),  \ldots,\\
  &\Gamma_{i(i\!-\!1)}(x)=\big[\sum_{j=1}^{i-2}q_{ij}(x),\sum_{j=1}^{i-1} q_{ij}(x)\big),
\end{array}
\end{equation}
and
\begin{align}\notag
\Gamma_{i(i+c_i )}(x)&=[-q_{i(i+c_i )}(x), 0),\\\label{a-8}
 \Gamma_{i(i+c_i -1)}(x) &  =\big[-q_{i(i+c_i -1)}(x)-q_{i(i+c_i )}(x),   -q_{i(i+c_i )}(x)\big),\\ \notag
 \ldots,\ \Gamma_{i(i+1)}(x)& = \big[-\sum_{j=i+1}^{i+c_i } q_{ij}(x), -\sum_{j=i+2}^{i+c_i } q_{ij}(x)\big),
\end{align}
where $c_i$ is given in (Q2). Analogously, by replacing $q_{ij}(x)$ in \eqref{a-7} and \eqref{a-8} with $\bar q_{ij}$ and $q_{ij}^\ast$ respectively, we can define the intervals $\bar\Gamma_{ij}$ and $\Gamma_{ij}^\ast$. Here and in the sequel, we put $\Gamma_{ij}(x)=\emptyset$ if $q_{ij}(x)=0$ and $\Gamma_{ii}(x)=\emptyset$ for the convenience of notation.
Such convention applies also to intervals $\bar \Gamma_{ij}$ and $\Gamma_{ij}^\ast$.

The sort order of intervals $\Gamma_{ij}(x)$, $\bar \Gamma_{ij}$ and $\Gamma_{ij}^\ast$ according to $j>i$ or $j<i$ will play important role in the argument below. The assumption on the existence of $c_i$ is used here such that on the negative half line, the first interval starting from $0$ is associated with
the state $j\in\S$ satisfying $j=\max\{k\in \S; k>i, q_{ik}(x)>0\}
$.  The common starting point $0$ of $\Gamma_{ij}(x)$, $\bar\Gamma_{ij}$ and $\Gamma_{ij}^\ast$ for different $i\in \S$ also plays an important role in our construction of the order-preservation coupling process.

\noindent\textbf{Step 2.} \emph{Explicit construction of Poisson random measure}
Here we use the method of \cite[Chapter I, p.44]{IW} to present a concrete construction of the Poisson random measure. Denote by $\mathbf{m}(\d x)$ the Lebesgue measure over $\R$. Let
$\xi_k$, $k=1,2,\ldots$, be random variables taking values in $[-K_0,K_0]$ with
\[\p(\xi_k \in \d x)= \frac{\mathbf{m}(\d x)}{2K_0},\]
and $\tau_k$, $k=1,2,\ldots$, be non-negative random variables such that
\[\p(\tau_k>t)=\exp [-2tK_0],\quad t\geq 0.\]
Assume that all $\xi_k$, $\tau_k$, $k\geq 1$, are mutually independent.
Let $\zeta_n =\tau_1 +\tau_2 +\cdots+\tau_n $, $ n=1,2,\ldots$, and $\zeta_0 =0$,
\begin{equation}\label{g-1}
\mathcal D_{\mathbf{p}}= \bigcup_{n\geq 1}\big\{\zeta_n \big\},
\end{equation}
and
\begin{equation}\label{g-2}
\mathbf{p}(t)=\!\!\sum_{0\leq s<t}\!\!\Delta \mathbf{p}(s), \quad   \Delta \mathbf{p}(s)=0\ \text{for $s\not\in \mathcal D_{\mathbf{p}}$},\ \text{and} \ \Delta\mathbf{p}(\zeta_n )=\xi_n ,\quad  n\geq 1,
\end{equation}
where $\Delta\mathbf{p}(s):=\mathbf{p}(s) -\mathbf{p}(s-)$.
The finiteness of $K_0$ means that $\lim_{n\to \infty}\zeta_n=\infty$ a.s., that is, during each finite time period, there exists a finite number of jumps for $(\mathbf{p}(t))$.
Let
\begin{equation}\label{g-3}
\mathcal{N}_{\mathbf{p}}([0,t]\!\times\! A)=\#\{s\in \!\mathcal D_{\mathbf{p}};  0\!\leq\! s\!\leq\! t, \Delta \mathbf{p}(s)\!\in\! A\},\  t>0,\ A\!\in\! \mathscr{B}(\R).
\end{equation}
As a consequence, $\mathbf{p}(t)$ and $\mathcal{N}_{\mathbf{p}}(\d t,\d z)$ are respectively a Poisson point process and a Poisson random measure with intensity measure $\d t\,\mathbf{m}(\d x)$.
It is always assumed that $\mathbf{p}(t)$ is independent of the Brownian motion $(B_t)$ in \eqref{a-1}.

\noindent\textbf{Step 3.} \emph{Construction of coupling processes}
Define three functions $\vartheta$, $\vartheta^\ast$ and $\bar\vartheta$ as follows:
\begin{align*}
  \vartheta(x,i,z)&=\sum_{j\in\S,j\neq i}(j-i)\mathbf{1}_{\Gamma_{ij}(x)}(z),\\
  \vartheta^\ast(i,z)&=\sum_{j\in\S,j\neq i}(j-i)\mathbf{1}_{\Gamma_{ij}^\ast}(z),\\
  \bar \vartheta(i,z)&=\sum_{j\in \S,j\neq i}(j-i)\mathbf{1}_{\bar \Gamma_{ij}}(z).
\end{align*}
Then, consider the following SDEs:
\begin{align}\label{g-4}
  \d \La_t&=\int_{\R}\!\vartheta(X_t,\La_{t-},z)\mathcal{N}_{\mathbf{p}}(\d t,\d z),\\ \label{g-5}
  \d  \bar\La_t&=\int_{\R}\!\bar\vartheta(\bar\La_{t-},z)\mathcal{N}_{\mathbf{p}}(\d t,\d z),\\ \label{g-6}
  \d \La_t^\ast &=\int_{\R}\!\vartheta^\ast(\La_{t-}^\ast,z)\mathcal{N}_{\mathbf{p}}(\d t,\d z),
\end{align} and $\La_0=\bar \La_0=\La_0^\ast=i_0\in\S$. Here recall that $(X_t)$ satisfies  \eqref{a-1}.

The fact that the solution to \eqref{g-4} satisfies \eqref{a-2} can be checked directly using  the property that
\[\p(\mathcal{N}_{\mathbf{p}}((0,\delta]\times A)\geq 2)=o(\delta),\quad \delta>0.\]
Next, we verify $(\bar \La_t)$ and $(\La_t^\ast)$ given by \eqref{g-5} and \eqref{g-6} are the jumping processes associated with $(\bar q_{ij})$ and $(q_{ij}^\ast)$ respectively.
According to It\^o's formula, for any bounded measurable function $F$ on $\S$,
  \begin{align*}
    \E F(\bar\La_t)&=F(\bar \La_0)+
    \E\int_0^t\!\int_{\R} \big(F(\bar\La_{s-}+ \bar\vartheta(\bar\La_{s-},z))-F(\bar\La_{s-})\big)
    \mathcal{N}_{\mathbf{p}}(\d s,\d z)\\
    &=F(\bar \La_0)+\E\int_0^t\!\int_{\R}\! \sum_{j\in \S}\big(F(j)-F(\bar\La_{s-})\big) \mathbf{1}_{\bar\Gamma_{\bar \La_{s-}j}}(z)\d s\m(\d z)\\
    &=F(\bar \La_0)+\E\int_0^t\! \sum_{j\in\S} \bar q_{\bar\La_{s-}j}\big(F(j)-F(\bar \La_{s-})\big)\d s.
  \end{align*}
  Denoted by $\bar P_t F(i_0)=\E F(\bar\La_t)$ with $\bar\La_0=i_0$, we obtain from above integral equation that
  \begin{equation}\label{ck-1}
  \bar P_tF(i_0)=F(i_0)+\int_0^t \bar P_s\big(\bar QF\big)\d s,
  \end{equation}
  where $\bar QF(i)=\sum_{j\in \S,j\neq i} \bar q_{ij}(F(j)-F(i))$, and the corresponding differential form is
  \begin{equation}\label{ck-2}
  \frac{\d {\bar P}_t F}{\d t}= \bar P_t \bar QF.
  \end{equation}
  Due to (Q1), the Kolmogorov forward equation \eqref{ck-2} admits a unique solution, and hence $(\bar\La_t)$ is a continuous-time Markov chain with transition rate matrix $\bar Q=(\bar q_{ij})$(cf. \cite[Corollary 2.24]{Chen}). The corresponding conclusion for $(\La_t^\ast)$ can be proved by the same method.

Consequently, through the previous three steps, we have completed the construction of the desired Markov chains $(\La_t^\ast)$ and $(\bar \La_t)$ used in Theorem \ref{main-1}. \fin

\begin{myrem}
  Our constructed coupling process $(X_t, \La_t,  \La_t^\ast, \bar \La_t)$ in terms of a common Poisson random measure presents a good order relation, and is not restricted to be of birth-death type. Let us compare it with the coupling constructed in \cite{CH}(presented in the argument of \cite[Lemma 3.9]{CH}). In \cite{CH}, the constructed Markov chain $(L_t)$ and the original jumping process $(\La_t)$ will move independently from each other until the time when they meet together. During their meeting time, the coupling is designed in a way such that $\La_t\geq L_t$. After the meeting time, the two processes $\La_t$ and $L_t$ locate at  different states and then they move independently once again until the next meeting time. However, the restriction of jumping in a  birth-death type could ensure that $\La_t \geq L_t$ after the meeting time.
\end{myrem}

\subsection{Proofs of the comparison theorem and its application}
Let us first present the proof of the comparison theorem.

\noindent\textbf{Proof of Theorem \ref{main-1}}. We only give out the proof of $\La_t\leq \bar\La_t$ for all $t\geq 0$ almost surely, and the corresponding solution for $\La_t$ and $\La_t^\ast$ can be proved in the same way.

According to the representation of \eqref{g-4} and \eqref{g-5}, the processes $(\La_t)$ and $(\bar \La_t)$ have no jumps outside $\mathcal{D}_{\mathbf{p}}$. So,   we only need to prove
\begin{equation}\label{dp}
\p\big(\La_t\leq \bar \La_t,\quad t\in \mathcal{D}_{\mathbf{p}}\big)=1.
\end{equation}

By the definition of $\bar q_{ij}$, it holds that
\begin{gather*}
  \text{for $j>i$},\ \ \bar q_{ij}\geq q_{lj}(x),\ \ \forall\,1\leq l\leq i, \ \forall\, x\in \R^d;\\
  \text{for $j<i$},\ \ \bar q_{ij}\leq q_{lj}(x),\ \ \forall\, j< l\leq i, \ \forall\, x\in \R^d.
\end{gather*}
By the sort order of the intervals $\Gamma_{ij}(x)$ and $\bar \Gamma_{ij}$, it holds that, for $i<k\in \S$,
\begin{equation}\label{g-7}
\bigcup_{r\geq m}\Gamma_{ir}(x)\subset \bigcup_{r\geq m}\bar \Gamma_{kr},\quad \ \forall\,m>k,\ \forall\,x\in \R^d,
\end{equation}
and
\begin{equation}\label{g-8}
\bigcup_{r\leq m}\Gamma_{ir}(x)\supset \bigcup_{r\leq m}\bar \Gamma_{kr},\quad \ \forall\, m<i, \ \forall\,x\in\R^d.
\end{equation}

Assuming  $i=\La_{\zeta_{n-1}}\leq \bar \La_{\zeta_{n-1}}=k$ for some $n\geq 1$, we are going to show that $\La_{\zeta_{n}}\leq \bar \La_{\zeta_n}$, whose proof is divided into four cases.

\noindent (i)\ If $\La_{\zeta_n}=m\geq k$, by \eqref{g-4} and the construction of the Poisson random measure $\mathcal{N}_{\mathbf{p}}(\d t,\d z)$, one gets
\[\xi_n\in \Gamma_{im}(X_{\zeta_n}).\]
By \eqref{g-7}, this yields that
\[\xi_n\in \bigcup_{r\geq m}\bar \Gamma_{kr}.\]
Together with \eqref{g-5}, this means that $\bar\La_{\zeta_n}$ must jump into the set $\{l\in\S;\ l\geq m\}$. Whence, $\La_{\zeta_n}\leq \bar \La_{\zeta_n}$.

\noindent (ii)\ If $\La_{\zeta_n}=m$ with $i<m<k$, \eqref{g-4} implies
\[\xi_n\in \Gamma_{im}(X_{\zeta_n}),\quad \text{and hence}\ \xi_n<0.\]
So,
\[\xi_n\not\in \bigcup_{j\leq k}\bar \Gamma_{kj}\subset [0,\infty),\]
which means that $\bar\La_{\zeta_n}$ cannot jump into the set $\{j\in\S; j\leq k\}$, and hence $\La_{\zeta_n}\leq \bar \La_{\zeta_n}$.  But, if $\La_{\zeta_n}=m$ with $m\leq i$ and $\bar\La_{\zeta_n}\leq i$, this situation is studied in the next case (iii).

\noindent (iii)\ If $\La_{\zeta_n}=m$ with $m\leq i$  and $\bar\La_{\zeta_n}>i$, it holds obviously that $\La_{\zeta_n}\leq \bar\La_{\zeta_n}$. If $\bar \La_{\zeta_n}=m'\leq i$, \eqref{g-5} and \eqref{g-8} yield that
\[\xi_n\in \bar\Gamma_{km'}\subset \bigcup_{r\leq m'}\Gamma_{ir}(X_{\zeta_n}). \]
Whence, $\La_{\zeta_n}$ jumps into $\{j\in \S; j\leq m'\}$. Hence, it still holds $\La_{\zeta_n}\leq \bar \La_{\zeta_n}$.

\noindent (iv)\ If $\bar \La_{\zeta_n}=m$ with $i<m<k$, then
\[\xi_n\in \bar \Gamma_{km},\ \xi_n>0,\ \ \text{and}\ \xi_n\notin \bigcup_{j>i}\Gamma_{ij}(X_{\zeta_n}).
\] So, it holds $\La_{\zeta_n}\leq i<m=\bar \La_{\zeta_n}$.

Consequently, if $\La_{\zeta_{n-1}}\leq \bar \La_{\zeta_{n-1}}$, we have $\La_{\zeta_n}\leq \bar \La_{\zeta_n}$ for $n\geq 1$.   By induction on $n$, we prove \eqref{dp} holds and finally
\[\p\big(\La_t\leq \bar \La_t,\ t\geq 0\big)=1.\]
The proof of Theorem \ref{main-1} is complete. \fin

\noindent \textbf{Proof of Theorem \ref{app-1}}
  For each $i\in\S$, denote by $(X_t^{(i)})$ the solution to the SDE
  \[\d X_t^{(i)}=b(X_t^{(i)}, i)\d t+\sigma(X_t^{(i)},i)\d B_t,\quad X_0^{(i)}=x,
  \] and $(P_t^{(i)})$ its associated semigroup.
  According to It\^o's formula and Gronwall's inequality, it follows from \eqref{g-9} that
  \[P_t^{(i)}\rho (x)=\E\rho (X_t^{(i)})\leq \e^{ \beta_i t} \rho (x). \]
  Let $\zeta_n$, $\mathcal{N}_{\mathbf{p}}(\d t,\d z)$, and $(\bar \La_t)$ be defined as those in the beginning of this section. Let $\E^{\mathcal{N}_{\mathbf{p}}}[\,\cdot\,] =\E[\,\cdot\,|\mathscr{F}^{ \mathcal{N}_{\mathbf{p}}}]$ be the conditional expectation w.r.t. the $\sigma$-algebra $\mathscr{F}^{\mathcal{N}_{\mathbf{p}}}=\sigma\{\mathbf{p}(s);\ s\geq 0\}$.
  The mutual independence of $\mathcal{N}_{\mathbf{p}}$ and $(B_t)$ yields that
  \begin{align*}
    \E^{\mathcal{N}_{\mathbf{p}}}[\rho(X_{\zeta_n})]&\leq \E^{\mathcal{N}_{\mathbf{p}}} [\rho(X_{\zeta_{n-1}})]
    +\E^{\mathcal{N}_{\mathbf{p}}} \Big[\int_{\zeta_{n-1}}^{\zeta_n} \!\beta_{\La_{\zeta_{n-1}}}\rho(X_s)\d s\Big].
  \end{align*}
  By Theorem \ref{main-1}, $\beta_{\La_s}\leq \beta_{\bar \La_s}$ a.s.. Furthermore, since $(\bar\La_s)$ depends only on $\mathcal{N}_{\mathbf{p}}$, we have
  \begin{equation}\label{g-12}
  \begin{split}
  \E^{\mathcal{N}_{\mathbf{p}}}[\rho(X_{\zeta_n})]&\leq  \E^{\mathcal{N}_{\mathbf{p}}}[\rho(X_{\zeta_{n-1}})]+\beta_{\bar\La_{\zeta_{n-1}}} \E^{\mathcal{N}_{\mathbf{p}}}\Big[\int_{\zeta_{n-1}}^{\zeta_n}\! \rho(X_s)\d s\Big].
  \end{split}
  \end{equation}
  Then, as $\bar \La_s\equiv \bar \La_{\zeta_{n-1}}$ for $s\in [\zeta_{n-1},\zeta_n)$ by \eqref{g-5},
  \begin{equation}\label{g-13}
    \E^{\mathcal{N}_{\mathbf{p}}} [\rho(X_{\zeta_n})]\leq \e^{\int_{\zeta_{n-1}}^{\zeta_n} \beta_{\bar \La_s}\d s} \E^{\mathcal{N}_{\mathbf{p}}}[\rho(X_{\zeta_{n-1}})].
  \end{equation}
  Setting $m_t:=\sup\{n\in\N;\ \zeta_n\leq t\}$, deducing recursively, we obtain
  \begin{align*}
     \E^{\mathcal{N}_{\mathbf{p}}} [\rho(X_t)]&\leq \e^{\int_{\zeta_{m_t}}^t\! \beta_{\bar \La_s}\d s}\E^{\mathcal{N}_{\mathbf{p}}} [\rho(X_{\zeta_{m_t}})]\\
     &\leq \e^{\int_{\zeta_{m_t}}^t\beta_{\bar \La_s}\d s}\e^{\int_{\zeta_{m_t-1}}^{\zeta_{m_t}} \!\!\beta_{\bar \La_s}\d s} \E^{\mathcal{N}_{\mathbf{p}}} [\rho(X_{\zeta_{m_t-1}})]\\
     &\leq \cdots\leq \e^{\int_0^t\beta_{\bar \La_s}\d s}\rho(x).
  \end{align*}
 Consequently,
 \begin{equation*}
 \E[\rho(X_t)]\leq \E\Big[\e^{\int_0^t\beta_{\bar \La_s}\d s}\Big]\rho(x),
 \end{equation*}
 and further by H\"older's inequality, for $p'\in (0,1]$,
 \begin{equation}\label{g-14}
 \E[\rho^{p'}(X_t)]\leq \E\Big[\e^{\int_0^tp'\beta_{\bar \La_s}\d s}\Big]\rho^{p'}(x).
 \end{equation}
 According to
 \cite[Propositions 4.1, 4.2]{Bar}, \eqref{g-10} yields that there exist  $p'\in (0,1]$, $C,\,\eta >0$ such that
 \[\E\Big[\e^{\int_0^t\!p'\beta_{\bar \La_s}\d s}\Big]\leq C\e^{-\eta t}.\]
 Combining this with \eqref{g-9}, \eqref{g-14}, the desired conclusion follows immediately.
\fin

\section{Perturbation of continuous-time Markov chain}

Given two transition rate matrices $Q=(q_{ij})_{i,j\in\S}$ and $\wt Q=(\tilde q_{ij})_{i,j\in\S}$ on $\S=\{1,2,\ldots,N\}$, $2\leq N\leq \infty$, there are two continuous-time Markov chains $(\La_t)$ and $(\tilde \La_t)$  associated with $Q$ and $\wt Q$ respectively  with $\La_0=\tilde \La_0$, the purpose of this section is to estimate the following quantity
\begin{equation}\label{e-1}
\frac 1t\int_0^t\p(\La_s\neq \tilde \La_s)\d s,\quad t>0,
\end{equation} in terms of the difference between $Q$ and $\wt Q$. The quantity \eqref{e-1} plays important role in the study of regime-switching processes. For instance, in \cite{SY}, it is the key point to show the stability of the process $(X_t)$ under the perturbation of $Q$. In \cite{Sh18}, it is the basis for proving the Euler-Maruyama's approximation of state-dependent regime-switching processes. Also, as mentioned in the introduction, it was used in \cite{SZ21} to study the smooth dependence of initial values for state-dependent \textbf{RSP}s. In this section, we shall improve the estimate of the quantity \eqref{e-1}, and apply it to develop the perturbation theory associated with the regime-switching processes.

In the classical perturbation theory of Markov chains, there are a lot of researches on the upper estimation of the total variation distance
$\|P_t-\wt P_t\|_{\var} $ between the semigroups $P_t$ and $\wt P_t$ associated with $(\La_t)$ and $(\tilde \La_t)$ respectively; see, e.g. \cite{DK, Mi03, Mi04, Mi05} and references therein. For instance, Mitrophanov \cite{Mi03, Mi04} showed that
\begin{equation}\label{e-2}
\|P_t-\wt P_t\|_{\mathrm{var}}\leq \frac{\e\tau_1}{\e-1}\|Q-\wt Q\|_{\ell_1},
\end{equation}
where
\[\beta(t)=\frac 12 \max_{i,j\in\S}\|(e_i-e_j)\exp(tQ)\|_{\ell_1},\ \ {\tau_1}=\inf\{t>0;\ \beta(t)\leq \e^{-1}\}.\] Recall that  the $\ell_1$-norm of a matrix $A=(a_{ij})_{i,j\in\S}$ is defined by $\|A\|_{\ell_1}=\sup_{i\in\S}\sum_{j\in\S}|a_{ij}|$, and the total variation distance between any two probability measures $\mu$ and $\nu$ on $\S$ is defined by
\[\|\mu-\nu\|_{\var}=\sup_{|h|\leq 1}\big|\sum_{i\in\S} h_i(\mu_i-\nu_i)\big|.\]
However, the estimate \eqref{e-2} and its method to establishing it is not applicable to \eqref{e-1}. Moreover, notice that
\[\|P_t-\wt P_t\|_{\var}=2\inf\p(\xi\neq \tilde \xi)\leq 2\p(\La_t\neq \tilde \La_t),\]
where the infimum is over all couplings $(\xi, \tilde \xi)$ with marginal distribution $P_t$ and $\wt P_t$ respectively.
The method in
\cite{DK, Mi03, Mi04} cannot be extended to dealing with \eqref{e-1} because the total variation norm plays an essential role in establishing \eqref{e-2}. In \cite{SY}, we provided an estimate of \eqref{e-1} through constructing a coupling process $(\La_t,\tilde \La_t)$ using Skorokhod's representation, where the intervals were constructed as in \cite{Gho}. Consequently,    \cite{SY} can only cope with jumping processes in a finite state space and the obtained estimate is not satisfactory especially for large time $t$.


In this section, we shall use the following assumptions:
\begin{itemize}
  \item[(H1)] $K_0:=\sup\{q_i,\tilde q_j;\, i,j\in\S\}<\infty$.
  \item[(H2)] There exists a $c_0\in\N$ such that $q_{ij}=\tilde q_{ij}=0$ for all $i,\,j\in\S$ with $|j-i|>c_0$.
\end{itemize}

\begin{mythm}\label{main-2}
Assume that $\mathrm{(H1)}$ and $\mathrm{(H2)}$ hold, then  there exist processes $(\La_t)$ and $(\tilde \La_t)$ such that for all $t>0$
\begin{equation}\label{h-2}
\frac{1}{t}\int_0^t\!\p(\La_s\neq \tilde \La_s)\d s\leq 1-\frac{1}{t\|Q-\wt Q\|_{\ell_1}}\Big(1-\e^{-\|Q-\wt Q\|_{\ell_1}t}\Big),
\end{equation}
which implies   that
\begin{equation}\label{h-2.5}
\frac 1 t\int_0^t\!\p(\La_s\neq \tilde \La_s)\d s \leq \min\Big\{\frac 12\|Q-\wt Q\|_{\ell_1} t, 1\Big\}.
\end{equation}
Also, for all $t>0$,
\begin{equation}\label{h-3}
  \begin{split}
   & \frac{1}{t}\int_0^t\!\p(\La_s\!\neq \!\tilde \La_s)\d s\geq \frac{\inf\limits_{i\in\S}\sum\limits_{j\neq i}|q_{ij}\!-\!\tilde q_{ij}|}{M+\|Q\!-\!\widetilde Q\|_{\ell_1}}\Big( 1\!-\!\frac{1}{(M\!+\!\|Q\!-\!\widetilde Q\|_{\ell_1})t}\big(1\!-\!\e^{-(M\!+\!\|Q\!-\!\widetilde Q\|_{\ell_1})t}\big)\Big),
  \end{split}
  \end{equation}
  where $M=4c_0K_0$.
\end{mythm}

\begin{myrem}\label{rem-2}
In \cite[Lemma 2.2]{SY}, given two transition rate matrices $Q$ and $\wt Q$ on a finite state space $\S=\{1,2,\ldots,N\}$,   a coupling process $(\La_t,\tilde \La_t)$ was constructed and satisfies
\begin{equation}\label{d-8}
\frac 1t\int_0^t \p(\La_s\neq \tilde \La_s)\d s\leq N^2 t\|Q-\wt Q\|_{\ell_1}.
\end{equation}
There is an important drawback  in \eqref{d-8}, that  is the appearance of $N^2$ on the right-hand side of \eqref{d-8}, which restricts the application of this result to the Markov chains on an infinite state space.   This drawback has been removed in Theorem \ref{main-2}, and further a lower estimate  of $\frac 1t\int_0^t\p(\La_s\neq \tilde \La_s)\d s$ is provided in current work.
\end{myrem}


\subsection{Construction of the coupling process}

In this part, we shall introduce the coupling process $(\La_t,\tilde \La_t)$ on $\S\times \S$ that will be used in the proof of Theorem \ref{main-2}. Similar to Section 2, it is also constructed using the Skorokhod representation theorem.   However, there are several subtle differences to fit the current purpose.

\noindent\textbf{Step 1}.\ \emph{Construction of intervals}\
Due to (H1) and (H2), we let
\begin{equation}\label{b-1}
\Gamma_{1k}=[(k-2)K_0,(k-2)K_0+q_{1k}),\quad \wt \Gamma_{1k}=[(k-2)K_0,(k-2)K_0+\tilde q_{1k})
\end{equation} for $2\leq k\leq c_0+1$,
and
\begin{equation}\label{b-2}
U_1=[0,c_0K_0).
\end{equation}
By (H1),  $q_{1k}\leq K_0$ and $\tilde q_{1k}\leq K_0$, and hence
\[\Gamma_{1k}\bigcap \Gamma_{1j}=\emptyset, \ \  \wt \Gamma_{1k}\bigcap \wt \Gamma_{1j}=\emptyset,\quad k\neq j.
\]
Moreover,
\[  \Gamma_{1k}\subset U_1,\quad  \wt \Gamma_{1k}\subset U_1,\ \ \forall\, 2\leq k\leq c_0+1.\]
For $n\geq 2$, define
\begin{equation}\label{b-5}
{\small\begin{split}
\Gamma_{nk}\!&=\!\Big[2(n\!-\!1)c_0K_0\!-\!(n\!-\!k)K_0, 2(n\!-\!1)c_0K_0\!-\!(n\!-\!k)K_0\!+\!q_{nk}\Big), \ \text{$n\!-\!c_0\!\leq k<n$,}  \\
\Gamma_{nk}\!&=\!\Big[2(n\!-\!1)c_0K_0\!+\!(k\!-\!n\!-\!1)K_0,
2(n\!-\!1)c_0K_0\!+\!(k\!-\!n\!-\!1)K_0\!+ \!q_{nk}\Big),\ \text{$n\!+\!c_0\!\geq k>n$}.
\end{split}}
\end{equation}
Define, similarly, $\wt \Gamma_{nk}$ by replacing $q_{nk}$ with $\tilde q_{nk}$ in \eqref{b-5}.
Let
\begin{equation}\label{b-6}
U_n=\Big[(2n-3)c_0 K_0, (2n-1)c_0 K_0\Big), \quad n\geq 2.
\end{equation} So $\Gamma_{nk}\subset U_n$ and $\wt \Gamma_{nk}\subset U_n$ for all $k,n\in\S$ with $|k-n|\leq c_0$.

Compared with the intervals constructed in \cite{Gho}, the starting point of each interval $\Gamma_{nk}$ constructed above does not depend on any other intervals $\Gamma_{ij}$, $i,\,j\in\S$ with $i\neq n$, $j\neq k$. We construct $\Gamma_{nk}$ in this way in order to remove the term $N^2$ appeared in \eqref{d-8}.

\noindent \textbf{Step 2.} \emph{Construction of Poisson random measure}
Denote by $\m(\d x)$ the Lebesgue measure over the real line $\R$.
Let $\xi_i^{(k)}$, $k,i=1,2,\ldots$, be $U_k$-valued random variables with
\[\p(\xi_i^{(k)}\in \d x)=\frac{\m(\d x)}{\m(U_k)},\]
and $\tau_i^{(k)}$, $k,i=1,2,\ldots$, be non-negative random variables such that
\[\p(\tau_i^{(k)}>t)=\exp[-t\m(U_k)],\quad t\geq 0.\]
Assume all $\xi_i^{(k)}$, $\tau_i^{(k)}$ to be mutually independent.
Denote by $\zeta_n^{(k)}=\tau_1^{(k)}+\ldots+\tau_n^{(k)}$ for $k,n\ge 1$ and $\zeta_0^{(k)}=0$ for $k\ge 1$. We define
\begin{equation}\label{b-7}
\mathcal D_{\tilde {\textbf{p}}}=\bigcup_{k\geq 1}\bigcup_{n\geq 0}\big\{\zeta_n^{(k)} \big\}
\end{equation}
and
\begin{equation}\label{b-8}
\tilde {\mathbf{p}}(t)=\sum_{0\leq s<t} \Delta\tilde{\mathbf{p}}(s),\quad \Delta \tilde{\mathbf{p}}(s)=0 \ \text{for $s\not\in \mathcal D_{\tilde{\mathbf{p}}}$}, \  \Delta\tilde{\mathbf{p}}(\zeta_n^{(k)})=\xi_n^{(k)},\quad k,n=1,2,\ldots.
\end{equation}
Let
\begin{equation}\label{b-9}
\mathcal{N}_{\tilde {\textbf{p}}}([0,t]\times A)=\#\big\{s\in \mathcal D_{\tilde{\mathbf{p}}}; \ 0< s\leq t, \tilde {\textbf{p}}(s)\in A\big\},\ \ t>0, A\in \mathscr{B}([0,\infty)).
\end{equation}
As a consequence, we get a Poisson random process $(\tilde{\mathbf{p}}(t))$ and its associated Poisson random measure $\mathcal{N}_{\tilde {\textbf{p}}}(\d t,\d x)$  with  intensity measure $\d t\,\m(\d x)$.

The construction of $({\tilde {\textbf{p}}}(t))$ is more complicated than that of $(\mathbf{p}(t))$ in   Section 2,  which is caused by the fact that the union of $\Gamma_{nk}$ for $n,k=1,2,\ldots$ may be unbounded.

\noindent\textbf{Step 3.} \emph{Construction of coupling  processes}
Define two functions $\vartheta$, $\tilde \vartheta$ associated with $\Gamma_{ij}$ and $\wt \Gamma_{ij}$, $i,j\in \S$, by
\begin{equation}\label{b-10}
\vartheta(i,z)=\sum_{j\in\S, j\neq i}(j-i)\mathbf{1}_{\Gamma_{ij}}(z),\ \ \tilde \vartheta(i,z)=\sum_{j\in\S, j\neq i}(j-i)\mathbf{1}_{\wt \Gamma_{ij}}(z).
\end{equation}
Then, the desired coupling process $(\La_t,\tilde \La_t)$ is given as the solutions of the following SDEs:
\begin{equation}\label{b-11}
\d \La_t=\int_{[0,\infty)} \vartheta(\La_{t-},z)\mathcal{N}_{\tilde{\mathbf{p}}}(\d t,\d z),\quad \La_0=i_0\in\S,
\end{equation}
and
\begin{equation}\label{b-12}
\d \tilde \La_t=\int_{[0,\infty)} \tilde \vartheta(\tilde \La_{t-},z)\mathcal{N}_{\tilde{\mathbf{p}}}(\d t,\d z),\quad \tilde \La_0=i_0\in\S.
\end{equation}
\begin{myrem}\label{rem-1}
The coupling process $(\La_t,\tilde \La_t)$ constructed above is different to the basic coupling of $(\La_t)$ and $(\tilde \La_t)$(see \cite[p.11]{Chen}). Indeed,
The transition rate matrix $\bar Q=(\bar{q}_{(ij)(k\ell)})_{i,j,k,\ell\in\S}$ of $(\La_t,\tilde \La_t)$ is given as follows: for $i,j,k, r\in\S$, which are different to each other,
\begin{equation}\label{q-c}
\begin{split}
\bar{q}_{(ij)(kr)}&=0,\quad \bar{q}_{(ij)(ik)}=\tilde q_{jk},\quad \bar{q}_{(ij)(kj)}=q_{ik},\\
\bar{q}_{(ii)(jj)}&=q_{ij}\wedge\tilde q_{ij},\quad \bar{q}_{(ii)(ji)}=(q_{ij}-\tilde q_{ij})\!\vee \!0,\quad \bar{q}_{(ii)(ij)}=(\tilde q_{ij}-q_{ij})\!\vee\! 0,
\end{split}
\end{equation} and  $\bar{q}_{(ii)(ii)}=-\!\sum\limits_{(\ell,\ell')\neq (i,i)}\!\!\bar{q}_{(ii)(\ell\ell')}$.

The transition rate matrix of the basic coupling is given by
  \begin{align*}
  q_{(ij)(kj)}&=(q_{ik}-\tilde q_{jk})\vee 0, \quad
  q_{(ij)(ik)} =(\tilde q_{jk}-q_{ik})\vee 0, \\
  q_{(ij)(kk)}&=q_{ik}\wedge \tilde q_{jk},
  \end{align*} for all $i,j,k\in\S$, and $i,j$ not necessary to be different.
\end{myrem}

By the previous construction of $\Gamma_{ij}$, $\wt \Gamma_{ij}$, and  $\mathcal{N}_p(\d t,\d x)$, the following  properties hold:
\begin{itemize}
  \item[(i)] The processes $(\La_t)$ and $(\tilde \La_t)$ could jump only when  the Poisson process $(\tilde{\mathbf{p}}(t))$ jumps.
  \item[(ii)] If $\La_t=\tilde \La_t=k$, for $\delta>0$, $\La_{t+\delta}\neq \tilde \La_{t+\delta}$ may happen only when $\zeta_n^{(k)}\in [t,t+\delta)$ and $\xi_n^{(k)}\in \Gamma_{kj}\Delta\wt \Gamma_{kj}$ for some $n\geq 1$, where $A\Delta B:=(A\backslash B)\cup(B\backslash A)$ for Borel sets $A$, $B$.
\end{itemize}

\subsection{Proof of Theorem \ref{main-2} and its application}\hfill

 \vspace*{0.5em}

 \textbf{Proof of Theorem \ref{main-2}}. \\
(i)\  \emph{Upper estimate} \
  For $\delta>0$, let
  \begin{align}\label{c-2}
  \alpha(\delta)&=\sup_{i\in \S}\p(\La_\delta\neq \tilde \La_\delta|\ \text{$\La_0=\tilde \La_0=i$}),\quad \quad
  \beta(\delta) = 1-\alpha(\delta).
  \end{align}
    Now we use the representation \eqref{b-11} and \eqref{b-12} of $(\La_t)$ and $(\tilde \La_t)$ to estimate $\alpha(\delta)$.
  Noting that  $\La_0=\tilde \La_0=i_0$ for some $i_0\in\S$, by \eqref{b-10}, \eqref{b-11} and \eqref{b-12},   we have
  \begin{equation}\label{c-4}
  \begin{split}
    \p(\La_\delta\neq \tilde \La_\delta)&=\p\big(\La_\delta\neq \tilde \La_{\delta}|\La_0=\tilde \La_0)=\p\big(\La_\delta\neq \tilde \La_\delta,\mathcal{N}_{\tilde{\mathbf{p}}}([0,\delta]\!\times\! U_{i_0})\geq 1\big)\\
    &=\p(\La_\delta\neq \tilde \La_\delta,\mathcal{N}_{\tilde{\mathbf{p}}}([0,\delta]\!\times\! U_{i_0})=1)+\p(\La_\delta\neq \tilde \La_\delta, \mathcal{N}_{\tilde{\mathbf{p}}}([0,\delta]\!\times\! U_{i_0})\geq 2).
  \end{split}
  \end{equation}
  Since  $H:=\m(U_{i_0})=2c_{0}K_0<\infty$, there exists a $C>0$ independent of the choice of $i_0\in\S$ such that
  \[\p\big(\mathcal{N}_{\tilde{\mathbf{p}}}([0,\delta]\!\times\! U_{i_0})\geq 2\big)=1-\e^{-H\delta}-H\delta\e^{-H\delta}\leq C\delta^2.\]
  Moreover,
  \begin{equation}\label{c-5}
  \begin{split}
    &\p(\La_\delta\!\neq \!\tilde \La_\delta, \mathcal{N}_{\tilde{\mathbf{p}}}([0,\delta]\!\times \! U_{i_0})\!=\!1)\\
    &=\!\int_0^\delta\!\p\big(\La_{\delta}\neq \tilde \La_\delta,\tau_1^{(i_0)}\!\in \d s,\tau_2^{(i_0)}\!>\delta-s\big)\\
    &=\int_0^\delta\!\p\big(\xi_1^{(i_0)}\! \in \!\cup_{j\in \S}\big(\Gamma_{i_0j}\Delta \widetilde \Gamma_{i_0j}\big),\tau_1^{(i_0)}\!\in\! \d s,\tau_2^{(i_0)}\!>\delta-s\big) \\
    &=\delta \e^{-H\delta}\sum_{j\in\S,j\neq i_0}|q_{i_0j}-\tilde q_{i_0j}|\\
    &\leq \delta\e^{-H\delta}\|Q-\widetilde Q\|_{\ell_1}\leq \delta \|Q-\widetilde Q\|_{\ell_1}.
  \end{split}
  \end{equation}
  In light of  \eqref{c-4} and \eqref{c-5}, we get
  \begin{equation}\label{c-5.5}
    \p(\La_\delta\neq \tilde \La_\delta)=\p(\La_\delta\neq \tilde \La_\delta|\La_0=\tilde \La_0=i_0)\leq \delta\|Q-\wt Q\|_{\ell_1}+C\delta^2,
  \end{equation}
  and hence
  \begin{equation}\label{c-6.5}
  \begin{split}
  \alpha(\delta)&\leq \delta\|Q-\wt Q\|_{\ell_1}+C\delta^2,\quad
  \beta(\delta)=1-\alpha(\delta) \geq 1-\delta\|Q-\wt Q\|_{\ell_1}-C\delta^2.
  \end{split}
  \end{equation}

  Next, let us consider the time $2\delta$.
  \begin{align*}
    \p(\La_{2\delta}\neq \tilde \La_{2\delta})&=\p(\La_{2\delta}\neq \tilde\La_{2\delta},\La_\delta= \tilde \La_{\delta})+\p(\La_{2\delta}\neq \tilde\La_{2\delta},\La_\delta\neq \tilde \La_\delta)\\
    &=\p(\La_{2\delta}\neq \tilde \La_{2\delta}\big|\La_\delta=\tilde \La_\delta)\p(\La_\delta=\tilde \La_\delta)\\
    &\quad+\p(\La_{2\delta}\neq \tilde \La_{2\delta}\big|\La_\delta\neq \tilde \La_{\delta})\p(\La_\delta\neq \tilde \La_{\delta})\\
    &\leq \alpha(\delta)\p(\La_\delta=\tilde \La_\delta)+\p(\La_\delta\neq \tilde \La_\delta)\\
    &= \alpha(\delta) +(1-\alpha(\delta))\p(\La_\delta\neq \tilde \La_\delta)\\
    &\leq \alpha(\delta)(1+\beta(\delta)),
  \end{align*} where we have used the time-homogeneity of Markov chain $(\La_t,\tilde \La_t)$ to get the  estimate
  $\dis \p(\La_{2\delta}\!\neq\! \tilde \La_{2\delta}\big|\La_{\delta}\!=\!\tilde \La_{\delta})\!\leq \! \alpha(\delta)$.
  Analogously, using the time-homogeneity of $(\La_t)$ and $(\tilde \La_t)$,
  \begin{align*}
    \p(\La_{3\delta}\neq \tilde \La_{3\delta})&\leq \p(\La_{3\delta}\neq \tilde\La_{3\delta}\big|\La_{2\delta}=\tilde \La_{2\delta})\p(\La_{2\delta}= \tilde \La_{2\delta})+\p(\La_{2\delta}\neq \tilde \La_{2\delta})\\
    &\leq \alpha(\delta)+ \beta(\delta) \p(\La_{2\delta} \neq \tilde\La_{2\delta}) \leq \alpha(\delta)(1+\beta(\delta)+\beta(\delta)^2)\\
    &=1-\beta(\delta)^3.
  \end{align*}
Deducing inductively, we get
  \begin{equation}\label{c-6}
    \p(\La_{k\delta}\neq \tilde \La_{k\delta})\leq \alpha(\delta)\sum_{m=1}^k\beta(\delta)^{m-1}
    = 1-\beta(\delta)^{k}   \quad \text{for}\ k\geq 4.
  \end{equation}
  Therefore, for $t>0$, setting $K(t)=\big[\frac{t}{\delta}\big]$, $t_k=k\delta$ for $k\leq K$ and $t_{K+1}=t$, by  \eqref{c-6.5},
  \begin{align*}
    &\int_0^t\p(\La_s\neq \tilde \La_s)\d s\\
    &\leq \sum_{k=0}^{K(t)}\! \int_{t_k}^{t_{k+1}}\! \!\!\big( \p(\La_s\neq \!\tilde \La_s,\La_{k\delta}\!=\!\tilde \La_{k\delta})+\p( \La_s\!\neq\! \tilde \La_s,\La_{k\delta}\!\neq \!\tilde \La_{k\delta})\big)\d s\\
    &\leq \sum_{k=0}^{K(t)} \int_{t_k}^{t_{k+1}} \!\!\p(\La_s\neq \tilde \La_s|\La_{k\delta}\neq \tilde \La_{k\delta} )\d s+  \sum_{k=0}^{K(t)}\!\int_{t_k}^{t_{k+1}}\!\p (\La_{k\delta}\neq \tilde \La_{k\delta})\d s\\
    &\leq \alpha(\delta) t+\delta\big(K(t)+1-\sum_{k=0}^{K(t)} \beta(\delta)^k\big)\\
    &\leq \alpha(\delta) t+\delta(K(t)+1)-\delta\sum_{k=0}^{K(t)}\! \big(1-\delta\|Q-\wt Q\|_{\ell_1}-C\delta^2\big)^k\\
    &\leq \alpha(\delta)t +\delta(K(t)+1)-\frac{1-\big(1-\delta\|Q -\wt Q\|_{\ell_1} -C \delta^2 \big)^{K(t)+1}} {\|Q-\wt Q\|_{\ell_1}+C\delta}.
    \end{align*}
  Letting  $\delta\downarrow 0$ and then dividing both sides by $t$,  we get the upper estimate
  \[\frac{1}{t}\int_0^t\!\p(\La_s\neq \tilde \La_s)\d s\leq 1-\frac{1}{t\|Q-\wt Q\|_{\ell_1}}\Big(1-\e^{-\|Q-\wt Q\|_{\ell_1}t}\Big)\leq 1.\]
  Using the inequality $\e^{-x}\leq 1-x+\frac 12 x^2$ for $x\geq 0$, we further get
  \[\frac{1}{t}\int_0^t\!\p(\La_s\neq \tilde \La_s)\d s\leq \min\Big\{\frac 12 \|Q-\wt Q\|_{\ell_1} t, 1\Big\}.\]
  Therefore, the upper estimates \eqref{h-2} and \eqref{h-2.5} have been proved.

\noindent (ii)\ \emph{Lower estimate}
We will estimate the difference by induction.

Due to \eqref{c-4},
  \begin{equation}\label{d-4}
  \begin{split}
  \p(\La_\delta\neq \tilde \La_\delta)&=\p(\La_\delta\neq \tilde \La_\delta|\La_0=\tilde \La_0)\geq \p(\La_{\delta}\neq \tilde \La_{\delta}, N_{\tilde{\mathbf{p}}}([0,\delta]\!\times\! U_{i_0})=1)\\
  &=\delta \e^{-H\delta}\sum_{j\neq i}|q_{i_0j}-\tilde q_{i_0j}|\\ &
  \geq \delta\e^{-H\delta}\inf_{i\in \S}\sum_{j\neq i}|q_{ij}-\tilde q_{ij}|
  =:\tilde \alpha(\delta).
  \end{split}
  \end{equation}
Then
  \begin{align*}
    &\p(\La_{2\delta}\neq \tilde \La_{2\delta})\\
    &=\p(\La_{2\delta}\neq \tilde \La_{2\delta},\La_{\delta}=\tilde\La_{\delta})+\p(\La_{2\delta}\neq \tilde \La_{2\delta},\La_\delta\neq \tilde \La_{\delta})\\
    &=\p(\La_{2\delta}\neq \tilde \La_{2\delta}\big|\La_{\delta}=\tilde \La_{\delta})(1-\p(\La_{\delta}\neq \tilde \La_{\delta}))\\
    &\quad +\p(\La_\delta\neq \tilde \La_{\delta})-\p(\La_{2\delta}=\tilde \La_{2\delta}\big|\La_{\delta}\neq \tilde \La_{\delta})\p(\La_\delta\neq \tilde \La_{\delta})\\
    &\geq \p(\La_{2\delta}\neq \La_{2\delta}\big|\La_\delta=\tilde \La_\delta)(1-\p(\La_{\delta}\neq \tilde \La_{\delta}))+\p(\La_\delta\neq \tilde \La_{\delta})\\
    &\quad-\p\big(\text{$\exists$\ jumps for $(\tilde{\mathbf{p}}(t))$ during $[\delta,2\delta)$ localizing in $U_{\La_\delta}\cup U_{\tilde \La_\delta}$}\big)\p(\La_\delta\neq \tilde \La_{\delta})\\
    &\geq \p(\La_{2\delta}\neq \La_{2\delta}\big|\La_\delta=\tilde \La_\delta)(1-\p(\La_{\delta}\neq \tilde \La_{\delta}))\\
    &\quad+\p(\La_\delta\neq \tilde \La_{\delta})-\big(1-\e^{-M\delta}\big)\p(\La_\delta\neq \tilde \La_\delta)\\
    &\geq \tilde \alpha(\delta)(1-\alpha(\delta))+\e^{-M\delta}\tilde \alpha(\delta)\\
    &\geq\tilde \alpha(\delta)+(\e^{-M\delta}-\delta\|Q-\widetilde Q\|_{\ell_1}-C\delta^2)\tilde \alpha(\delta),
  \end{align*}
  where we have used \eqref{c-5.5} and
  \[\p\big(N_{\tilde{\mathbf{p}}}\big([\delta, 2\delta]\!\times\!\big(U_{\La_\delta}\cup U_{\tilde \La_{\delta}}\big)\big)\geq 1\big)\leq 1- \e^{-M\delta},\] since $\m(U_{i}\cup U_{j})\leq M=4c_0K_0$ for all $i,\,j\in\S$.
  Setting \[\gamma(\delta)=\e^{-M\delta}-\delta\|Q-\widetilde Q\|_{\ell_1}-C\delta^2,\ \text{which is positive when $\delta$ is small enough},\] we rewrite the previous estimate in the form
  \begin{equation}\label{d-5}
  \p(\La_{2\delta}\neq \tilde \La_{2\delta})\geq \tilde \alpha(\delta)+\gamma(\delta)\tilde \alpha(\delta).
  \end{equation}
Repeating this procedure, we obtain that
  \begin{equation}\label{d-6}
  \p(\La_{k\delta}\neq \tilde \La_{k\delta})\geq \tilde \alpha(\delta)\sum_{m=1}^{k}\gamma(\delta)^{m-1},\quad k\geq 3.
  \end{equation}
Then, by \eqref{d-6},  for $K\in \N$,
  \begin{align*}
    \int_0^{K\delta}\!\!\!\p(\La_s\!\neq\! \tilde \La_s)\d s&=\sum_{k=0}^{K }\int_{k\delta}^{(k+1)\delta} \!\!\big(\p(\La_s\!\neq\!\tilde \La_s,\La_{k\delta}\!\neq \!\tilde \La_{ k\delta})+ \p(\La_s\!\neq\! \tilde \La_s,\La_{k\delta}= \tilde \La_{k\delta})\big)\d s\\
    &\geq \sum_{k=0}^{K }\int_{k\delta}^{(k+1)\delta} \!\! \big(\p(\La_{k\delta}\neq \tilde \La_{k\delta})-\p(\La_{k\delta}\neq \tilde \La_{k\delta}, \La_s=\tilde \La_s)\big)\d s\\
    &\geq \delta\sum_{k=1}^{K}\p(\La_{k\delta}\neq \tilde \La_{k\delta})-\delta(1-\e^{-M\delta}) \sum_{k=1}^{K}\p(\La_{k\delta}\neq \tilde \La_{k\delta})\\
    &\geq \delta \e^{-M\delta}\sum_{k=1}^K\tilde \alpha(\delta)\frac{1-\gamma(\delta)^k}{1-\gamma(\delta)}\\
    &=\e^{-M\delta}\frac{\delta\tilde \alpha(\delta)}{1-\gamma(\delta)}\big(K-\frac{\gamma(\delta)(1-\gamma(\delta)^K)}
    {1-\gamma(\delta)}\big).
  \end{align*}
  Since
  \[\lim_{\delta\downarrow 0}\frac{\tilde \alpha(\delta)}{1-\gamma(\delta)}=\frac{\inf\limits_{i\in\S}\sum\limits_{j\neq i}|q_{ij}-\tilde q_{ij}|}{M+\|Q-\widetilde Q\|_{\ell_1}},\qquad \lim_{\delta\downarrow 0} \gamma(\delta)^{\frac{t}{\delta}}=\e^{-(M+\|Q-\widetilde Q\|_{\ell_1}) t},  \]
  by taking $K=\big[\frac{t}{\delta}\big]$  in the previous estimation and letting $\delta$ downward to $0$, we finally get
  \begin{equation}\label{d-7}
  \int_0^t\p(\La_s\neq \tilde\La_s)\d s\geq \frac{\inf\limits_{i\in\S}\sum\limits_{j\neq i}|q_{ij}-\tilde q_{ij}|}{M+\|Q-\widetilde Q\|_{\ell_1}}\Big( t-\frac{1}{M\!+\!\|Q\!-\!\widetilde Q\|_{\ell_1}}\big(1-\e^{-(M\!+\!\|Q\!-\!\widetilde Q\|_{\ell_1})t}\big)\Big),
  \end{equation}
  which is the desired lower estimate. The proof is complete.\fin

Next, we consider the application of Theorem \ref{main-2}. First, let us consider its application to the perturbation theory on the invariant probability measures of continuous time Markov chains on infinite state spaces. There were many works on the perturbation theory of Markov chains, such as, \cite{Mi05}, \cite{ZI94} and references therein. We refer the readers to the recent review paper Zeifman, Korolev and Satin \cite{ZKS} for more discussions on this topic.

Let $P_t$ and $\wt P_t$ denote the semigroup w.r.t.\! the  transition rate matrix $Q$ and $\wt Q$ respectively. Assume that there exist invariant probability measures $\pi=(\pi_i)_{i\in\S}$ and $\tilde \pi=(\tilde \pi_i)_{i\in\S}$ associated respectively with $P_t$ and $\wt P_t$, i.e.
\[\pi P_t=\pi,\qquad \tilde \pi \wt P_t=\tilde   \pi.\]

\begin{mycor}\label{cor-1}
Assume  $\mathrm{(H1)}$ and  $\mathrm{(H2)}$ hold. Suppose that there exists a function $\eta:[0,\infty)\to [0,2]$ satisfying $\int_0^\infty \eta_s\d s<\infty$ such that for some $i_0\in\S$
\[\|P_t(i_0,\cdot)-\pi\|_{\var}\leq \eta_t,\quad \|\wt P_t(i_0,\cdot)-\tilde \pi\|_{\var}\leq \eta_t,\quad t\geq 0.\]
Then, it holds
\begin{equation}\label{p-1}
\|\pi-\tilde \pi\|_{\var}\leq 2\sqrt 2\Big( \int_0^\infty \eta_s\d s\Big)^{\frac 12}\sqrt{\|Q-\wt Q\|_{\ell_1}}.
\end{equation}
\end{mycor}

\begin{proof}
  For any bounded function $h$ on $\S$ with $|h|_\infty:=\sup_{i\in\S} |h_i|\leq 1$, it holds that
  \begin{align*}
    &|\pi(h)-\tilde \pi(h)|=\big|\sum_{i\in\S} \pi_i h_i-\sum_{i\in\S}\tilde \pi_i h_i\big|\\
    &\leq \big|\pi(h)-\frac 1t\!\int_0^t\! P_sh(i_0)\d s\big|\!+\!\big|\tilde \pi(h)\!-\!\frac 1t\!\int_0^t\!\wt P_s h(i_0)\d s\big|\!+\!\big|\frac 1t \int_0^t \!P_s h(i_0)\!-\!\wt P_s h(i_0)\d s\big|\\
    &\leq \frac 1t \int_0^t|P_s h(i_0)-\pi(h)|\d s+\frac 1t \int_0^t\!|\wt P_s h(i_0)-\tilde \pi(h)|\d s +\|Q-\wt Q\|_{\ell_1} t\\
    &\leq \frac 2 t\int_0^\infty \eta_s\d s+\|Q-\wt Q\|_{\ell_1} t,\qquad \forall t>0,
  \end{align*}
  where we have used \eqref{h-2.5} in Theorem \ref{main-2}. By  taking $t=\Big(2\int_0^\infty \eta_s\d s/\|Q-\wt Q\|_{\ell_1}\Big)^{1/2}$, we arrive at
  \[\|\pi-\tilde \pi\|_{\var}=\sup_{|h|\leq 1} |\pi(h)-\tilde\pi(h)|\leq 2\sqrt 2\Big( \int_0^\infty \eta_s\d s\Big)^{\frac 12}\sqrt{\|Q-\wt Q\|_{\ell_1}},
  \]
  which is the desired conclusion.
\end{proof}

\begin{myrem}
  When $\S$ is a finite state space, the stability of $\pi$ in terms of the perturbation of $Q$ has been studied in Freidlin and Wentzell \cite{FW} and Faggionato et al. \cite{FGC}. \cite{FW} proved it through expressing $\pi$ as a polynomial of transition probabilities. This result was applied in \cite{BDG18} to establish the averaging principle for multiple time-scale systems. \cite{FGC} proved a similar result by Perron-Frobenius Theorem to express $\pi$ in terms of a nonzero right eigenvector of $Q$-matrix with eigenvalue 0.
\end{myrem}

Second, we can apply Theorem \ref{main-2} to improve all the main results  \cite[Theorems 1.1-1.4]{SY}. Here we only state the improvement of \cite[Theorem 1.1]{SY} to save space.

Consider the regime-switching system $(X_t,\La_t)$ satisfying
\begin{equation}\label{d-9}
\d X_t=b(X_t,\La_t)\d t+\sigma(X_t,\La_t)\d B_t,\quad X_0=x_0\in\R^d,\ \La_0=i_0\in\S,
\end{equation}
and $(\La_t)$ is a Markov chain on $\S=\{1,2,\ldots,N\}$, $2\leq N\leq \infty$.
In realistic application, sometimes one can only get an estimation $\wt Q$ of the original transition rate matrix $Q$ of $(\La_t)$. $\wt Q$ determines another Markov chain $(\tilde \La_t)$. Correspondingly, the studied system $(X_t)$ turns into $(\wt X_t)$ satisfying
\[\d \wt X_t=b(\wt X_t,\tilde \La_t)\d t+\sigma(\wt X_t,\tilde \La_t)\d B_t,\ \ \wt X_0=x_0,\ \tilde \La_0=i_0.
\]
It is necessary to measure the difference between $X_t$ and $\wt X_t$ caused by the difference between $Q$ and $\wt Q$. We shall characterize the difference between $X_t$ and $\wt X_t$ via the Wasserstein distance between their distributions.

For any two probability measures $\mu$, $\nu$ on $\R^d$, define the $L_2$-Wasserstein distance between them by
\[W_2(\mu,\nu)^2=\inf_{\pi\in\mathscr{C}(\mu,\nu)}\Big\{\int_{\R^d\!\times\!\R^d}\!\!|x-y|^2\pi(\d x,\d y)\Big\},
\]
where $\mathscr{C}(\mu,\nu)$ denotes the set of all the couplings of $\mu$, $\nu$ on $\R^d\times \R^d$.

\begin{mycor}\label{tt-1}
Assume that $\mathrm{(A1)}$, $\mathrm{(A2)}$, $\mathrm{(H1)}$, and $\mathrm{(H2)}$ hold. Denote by $\mathcal{L}(X_t)$ and $\mathcal{L}(\wt X_t)$ the distribution of $X_t$ and $\wt X_t$ respectively. Then
\begin{equation}\label{d-10}
W_2(\mathcal{L}(X_t),\mathcal{L}(\wt X_t))^2\leq C(p,t)\big(\|Q-\wt Q\|_{\ell_1}\big)^{(p-1)/p},  \quad t>0,
\end{equation}
where $p>1$ and $C(p,t)$ is a positive constant depending only on $p$ and $t$.
\end{mycor}
\begin{proof}
  This result can be proved along the line of \cite[Theorem 1.1]{SY} by replacing the upper bound of \cite[Lemma 2.2]{SY} with the upper bound given in Theorem \ref{main-2}. The constant $C(p,t)$ can be explicitly expressed as in \cite{SY}.
\end{proof}

\end{document}